\documentclass{article}
\usepackage{natbib}
\usepackage{amsmath,amsthm,amsopn,amstext,amscd,amsfonts,amssymb,verbatim}
\def\diag{\mathop{\rm diag}\nolimits}

\def\build#1#2#3{\mathrel{\mathop{#1}\limits^{#2}_{#3}}}

\def\Cov{\mathop{\rm Cov}\nolimits}

\def\E{\mathop{\rm E}\nolimits}

\renewenvironment{abstract}
                 {\vspace{6pt}
                  \begin{center}
                  \begin{minipage}{5in}
                  \centerline{\textbf{Abstract}}
                  \noindent\ignorespaces
                 }
                 {\end{minipage}\end{center}}
\newtheorem{thm}{\textbf{Theorem}}[section]
\newtheorem{cor}{\textbf{Corollary}}[section]

\theoremstyle{definition}

\title{\Large \textbf{Response surface methodology: Asymptotic normality of the optimal solution}}
\author{
  \textbf{Jos\'e A. D\'{\i}az-Garc\'{\i}a} \thanks{Corresponding author\newline
   {\bf Key words.}  Multiple Response Surfaces (MRS), multiobjective optimisation, probabilistic optimisation,
     stochastic programming, matrix optimisation.\newline
    2000 Mathematical Subject Classification. 62K20; 90C15; 90C29}\\
  {\normalsize Department of Statistics and Computation} \\
  {\normalsize 25350 Buenavista, Saltillo, Coahuila, Mexico} \\
  {\normalsize E-mail: jadiaz@uaaan.mx} \\[2ex]
  \textbf{Jos\'e E. Rodr\'{\i}guez} \\
  {\normalsize Department of Mathematics} \\
  {\normalsize University of Guanajuato} \\
  {\normalsize Jalisco s/n, Mineral de Valenciana}\\
  {\normalsize 36240 Guanajuato, M\'exico}\\
  {\normalsize E-mail: elias.rodriguez@ugto.mx}\\[2ex]
  \textbf{Rogelio Ramos-Quiroga} \\
  {\normalsize Centro de Investigaci\'on en Matem\'aticas} \\
  {\normalsize Department of Probability and Statistics} \\
  {\normalsize Callej\'on de Jalisco s/n, Mineral de Valenciana}\\
  {\normalsize 36240 Guanajuato, Guanajuato, M\'exico}\\
  {\normalsize E-mail: rramosq@cimat.mx}
}
\date{}
\begin{document}
\maketitle

\begin{abstract}
Sensitivity analysis of the optimal solution in response surface methodology is studied and an
explicit form of the effect of perturbation of the regression coefficients on the optimal
solution is obtained. The characterisation of the critical point of the convex program
corresponding to the optimum of a response surface model is also studied. The asymptotic
normality of the optimal solution follows by standard methods.
\end{abstract}

\section{Introduction}\label{Sec1}

From a point of view of the mathematical programming, the sensitivity analysis studies the
effect of small perturbations in the parameters on the optimal objective function value and on
the critical point for mathematical programming problems. In general, these parameters shape
the objective function and constraint the approach to the problem of mathematical programming.
The sensitivity analysis of the mathematical programming has been studied by several authors,
see \citet{j:77}, \citet{d;84} and \citet{fg:82} among many other. As a direct consequence of
the sensitivity analysis, the asymptotic normality study of the critical point follows by
standard methods of mathematical statistics (see similar results for the case of maximum
likelihood estimates \citet{as:58}). This last consequence makes  the sensitivity analysis
very appealing for researching\footnote{In the context of the mathematical statistical the
sensitivity analysis consist in to study the ways in which the estimators of certain model are
affected by omission of a particular set of variables or by the inclusion or omission of a
particular observation or set of observations, see \citet{chh:88}.} from a statistical point of
view.

In this work, we study the effect of perturbations of the regression parameters on the optimal
solution of the response surface model and then the asymptotic normality of the critical point
is obtained.

A very useful statistical tool in the study of designs, phenomena and experiments, is the
response surfaces methodology. Which enables to find an analytical relationship between the
response and controlled variables through a process of continuous improvement and optimisation.

Specifically, it is assumed that a researcher knows a system, which considers some observable
response variables $y$ which depends on some input variables, $x_{1}, \dots x_{k}$. Also it is
assumed that the input  variables $x_{i}^{'s}$ can be controlled by the researcher with a
minimum error.

In general we have that
\begin{equation}\label{rrs}
    y(\mathbf{x}) = \eta(x_{1}, \dots x_{k})
\end{equation}
where the form of the function $\eta(\cdot)$, usually termed as the true response surface, is
unknown and perhaps, very complicated; and $\mathbf{x} =(x_{1}, \dots x_{k})'$. The success of
the response surfaces methodology depends on the approximation of $\eta(\cdot)$ for a
polynomial of low degree in some region.

In this work we assume that $\eta(\cdot)$ can be soundly approximated by a polynomial of second
order, that is
\begin{equation}\label{ersp}
   y(\mathbf{x}) = \beta_{0} + \displaystyle\sum_{i=1}^{n}\beta_{i}x_{i} + \sum_{i=1}^{n}
    \beta_{ii}x_{i}^{2} +  \sum_{i=1}^{n}\sum_{j>i}^{n}\beta_{ij}x_{i}x_{j}
\end{equation}
where the unknown parameters $\beta_{j}^{'s}$ can be estimated via regression's techniques, as
it is described in next section.

Next, we are interested in obtaining the levels of the input variables $x_{i}^{'s}$ such that
the response variables $y$ is minimum (optimal). This can be achieved if the following
mathematical program is solved
\begin{equation}\label{opp}
  \begin{array}{c}
    \build{\min}{}{\mathbf{x}}y(\mathbf{x}) \\
    \mbox{subject to}\\
    \mathbf{x} \in \mathfrak{X},
  \end{array}
\end{equation}
where $\mathfrak{X}$ is certain operating region for the input variables $x_{i}^{'s}$.

Now, two questions, intimately related, can be observed:
\begin{enumerate}
  \item When the estimation of (\ref{ersp}) is considered into (\ref{opp}) the critical point
  $\mathbf{x}^{*}$ obtained as solution shall be a function of the estimators
  $\widehat{\beta}_{j}^{\ 's}$ of the $\beta_{j}^{'s}$. Thus, given that $\widehat{\beta}_{j}^{\ 's}$ are random variables,
  then $\mathbf{x}^{*}\equiv \mathbf{x}^{*}(\widehat{\beta}_{j}^{\ 's})$ is a random vector
  too. The question is, if the  distribution of $\widehat{\boldsymbol{\beta}}$ is known, then what is the
  distribution of $\mathbf{x}^{*}(\widehat{\beta}_{j}^{\ 's})$?
  \item And, perhaps it is not sufficient to know only a point estimate of $\mathbf{x}^{*}(\widehat{\beta}_{j}^{\ 's})$,
 should be more convenient to know an interval estimate.
\end{enumerate}
The distribution of the critical point in response surface methodology was studied by
\citet{dr:01,dre:02}, when $y(\mathbf{x})$ is defined as an hyperplane.

In this work we give a response to these two questions. First, in Section \ref{sec2} some
notation is established. In Section \ref{sec3} the response surface optimisation problem is
proposed. First-order and second-order Kuhn-Tucker conditions are stated characterising the
critical point in Section \ref{sec4}. Finally, the asymptotic normality of a critical point is
established in Section \ref{sec5}.

\section{Notation}\label{sec2}

A detailed discussion of response surface methodology may be found in \citet{kc:87} and
\citet{mma:09}. For convenience, their principal properties and usual notation shall be
restated here.

Let $N$ be the number of experimental runs. The response variable is measured for each setting
of a group  of $n$ coded variables (also are termed factors) $x_{1}, x_{2},\dots,x_{n}$. We
assume that the response variable can be modeled by \textit{second order polynomials regression
model} in terms of $x_{i}^{'s}$. Hence, the response model can be written as
\begin{equation}\label{lmu}
    \mathbf{y} = \mathbf{X}\boldsymbol{\beta} + \boldsymbol{\varepsilon}
\end{equation}
where $\mathbf{y} \in \Re^{N}$ is the vector of observations on the response variable,
$\mathbf{X} \in \Re^{N \times p}$ is an matrix of rank $p$ termed design or regression matrix,
$p = 1 + n + n(n+1)/2$, $\boldsymbol{\beta} \in \Re^{p}$ is a vector of unknown constant
parameters, and $\boldsymbol{\varepsilon} \in \Re^{N}$ is a random error vector such that
$\boldsymbol{\varepsilon} \sim \mathcal{N}_{N} (\mathbf{0}, \sigma^{2}\mathbf{I}_{N})$, i.e.
$\boldsymbol{\varepsilon}$ has an \textit{$N$-dimensional  normal distribution} with $\E
(\boldsymbol{\varepsilon}) = \mathbf{0}$ and $\Cov(\boldsymbol{\varepsilon}) =
\sigma^{2}\mathbf{I}_{N}$. Then we have:
\begin{description}
\item[$\centerdot$] $\mathbf{x} = \left(x_{1}, x_{2},\dots,x_{n}\right)'$: Vector of controllable
variables or factors. Formally, to each factor $A, B,...$ is associated an $x_{i}^{'s}$
variable

\item[$\centerdot$] $\widehat{\boldsymbol{\beta}}$: The least squares estimator of
$$
  \boldsymbol{\beta} = (\beta_{0},\beta_{1},\dots,\beta_{n},\beta_{11},\dots,\beta_{nn},\beta_{12},
  \dots,\beta_{(n-1)n})'
$$
given by
$$
  \widehat{\boldsymbol{\beta}}= (\mathbf{X}'\mathbf{X})^{-1}\mathbf{X}'\mathbf{y} =
  (\widehat{\beta}_{0},\widehat{\beta}_{1},\dots,\widehat{\beta}_{n},\widehat{\beta}_{11},
  \dots,\widehat{\beta}_{nn},\widehat{\beta}_{12},\dots,\widehat{\beta}_{(n-1)n})'.
$$
Furthermore, under the assumption that $\boldsymbol{\varepsilon} \sim \mathcal{N}_{N}
(\mathbf{0}, \sigma^{2}\mathbf{I}_{N})$, then $\widehat{\boldsymbol{\beta}} \sim
\mathcal{N}_{p} \left(\boldsymbol{\beta}, \sigma^{2}(\mathbf{X}'\mathbf{X})^{-1}\right)$.

\item[$\centerdot$] $\mathbf{z}(\mathbf{x}) = (1,x _{1}, x_{2}, \dots, x_{n}, x _{1}^{2}, x_{2}^{2},
\dots, x_{n}^{2}, x _{1} x_{2}, x_{1}x_{3}\dots, x_{n-1}x_{n})'.$

\item[$\centerdot$] $\widehat{\boldsymbol{\beta}}_{1}=
(\widehat{\beta}_{1},\dots,\widehat{\beta}_{n})'$ and
$$
  \widehat{\mathbf{B}} = \frac{1}{2}
  \left (
     \begin{array}{cccc}
       2\widehat{\beta}_{11} & \widehat{\beta}_{12} & \cdots & \widehat{\beta}_{1n} \\
       \widehat{\beta}_{21} & 2\widehat{\beta}_{22} & \cdots & \widehat{\beta}_{2n} \\
       \vdots & \vdots & \ddots & \vdots \\
       \widehat{\beta}_{n1} & \widehat{\beta}_{n2} & \cdots & 2\widehat{\beta}_{nn}
     \end{array}
  \right)
$$

\item[$\centerdot$] \hspace{-3cm} \vspace{-.85cm}
\begin{eqnarray*}
  \widehat{y} (\mathbf{x}) &=& \mathbf{z}'(\mathbf{x})\widehat{\boldsymbol{\beta}} \\
    &=& \widehat{\beta}_{0} + \displaystyle\sum_{i=1}^{n}\widehat{\beta}_{i}x_{i} + \sum_{i=1}^{n}
    \widehat{\beta}_{ii}x_{i}^{2} +
    \sum_{i=1}^{n}\sum_{j>i}^{n}\widehat{\beta}_{ij}x_{i}x_{j}\\
    &=& \widehat{\beta}_{0}+ \widehat{\boldsymbol{\beta}}'_{1}\mathbf{x}+\mathbf{x}^{'}\widehat{\mathbf{B}}\mathbf{x}:
\end{eqnarray*}
Estimated response surface or predictor equation at the point $\mathbf{x}$.

\item[$\centerdot$]$\widehat{\sigma}^{2} = \displaystyle \frac{\mathbf{y}'(\mathbf{I}_{N} -
\mathbf{X}(\mathbf{X}'\mathbf{X})^{-1}\mathbf{X}')\mathbf{y}}{N-p}$: Estimator of variance
$\sigma^{2}$ such that $(N-p)\widehat{\sigma}^{2}/\sigma^{2}$ has a $\chi^{2}$-distribution
with $(N-p)$ freedom degrees; denoted this fact as $(N-p)\widehat{\sigma}^{2}/\sigma^{2} \sim
\chi^{2}_{_{N-p}}$.

Finally, note that
\begin{equation}\label{eq4}
   \widehat{\Cov} (\widehat{\boldsymbol{\beta}}) = \widehat{\sigma}^{2}(\mathbf{X}'
    \mathbf{X})^{-1}.
\end{equation}
\end{description}

\section{Response surface optimisation}\label{sec3}

For convenience, those concepts and notations required here are listed below in terms of the
estimated model of response surface. Definitions and details properties, etc., may be found in
\citet{kc:87} and \citet{mma:09}.

Ideally, the goal would be to solve (\ref{opp}), however note that in general the parameters of
form $\beta_{j}^{'s}$ are unknown. Then proceeding as it is indicated in the previous section,
the $\beta_{j}^{'s}$ parameters are estimated. Then, the response surface optimisation problem
in general is proposed as the following substitute mathematical program
\begin{equation}\label{equiv1}
  \begin{array}{c}
    \build{\min}{}{\mathbf{x}}\widehat{y}(\mathbf{x}) \\
    \mbox{subject to}\\
    \mathbf{x} \in \mathfrak{X},
  \end{array}
\end{equation}
which is a nonlinear mathematical problem, see \citet{kc:87},  and
\citet{r:79}. Here $\mathfrak{X}$ denotes the experimental region,
which is defined, in general, as a hypercube
$$
  \mathfrak{X}= \{\mathbf{x}|l_{i} < x_{i} < u_{i}, \quad i = 1, 2, \dots, n\},
$$
where $\mathbf{l} = \left(l_{1}, l_{2}, \dots, l_{n}\right)'$
defines the vector of lower bounds of factors, and $\mathbf{u} =
\left(u_{1}, u_{2}, \dots, u_{n}\right)'$ defines the vector of
upper bounds of factors. Alternatively, the experimental region is
defined as an hypersphere
\begin{equation}\label{c}
    \mathfrak{X}= \{\mathbf{x}|\mathbf{x}'\mathbf{x} \leq c^{2},  c > 0\} ,
\end{equation}
where note that $\mathbf{x}'\mathbf{x} = ||\mathbf{x}||^{2}$ and in general $c$ is specified by
the experimental design model used. In this paper it is considered $\mathfrak{X}$ defined by
(\ref{c}).

\section{Characterisation of the critical point}\label{sec4}

Let $\mathbf{x}^{*}(\widehat{\boldsymbol{\beta}}) \in \Re^{n}$ be the unique optimal solution
of program (\ref{equiv1}) with the corresponding Lagrange multiplier
$\lambda^{*}(\widehat{\boldsymbol{\beta}}) \in \Re$. The Lagrangian is defined by
\begin{equation}\label{la}
    L(\mathbf{x}, \lambda; \widehat{\boldsymbol{\beta}}) = \widehat{y}(\mathbf{x}) + \lambda(||\mathbf{x}||^{2} -
  c^{2}).
\end{equation}
Similarly, $\mathbf{x}^{*}(\boldsymbol{\beta}) \in \Re^{n}$ denotes the unique optimal solution
of program (\ref{opp}) with the corresponding Lagrange multiplier
$\lambda^{*}(\boldsymbol{\beta}) \in \Re$.

Now we establish the local Kuhn-Tucker conditions that guarantee that the Kuhn-Tucker point
$\mathbf{r}^{*}(\widehat{\boldsymbol{\beta}}) =
\left[\mathbf{x}^{*}(\widehat{\boldsymbol{\beta}}),
\lambda^{*}(\widehat{\boldsymbol{\beta}})\right]' \in \Re^{n+1}$ is a unique global minimum of
convex program (\ref{equiv1}). First recall that for $f:\Re^{n} \rightarrow \Re$, then
$\displaystyle\frac{\partial f}{\partial \mathbf{x}} \equiv \nabla_{\mathbf{x}}$ denotes the
gradient of function $f$.

\begin{thm}\label{teo1}
The necessary and sufficient conditions that a point
$\mathbf{x}^{*}(\widehat{\boldsymbol{\beta}}) \in \Re^{n}$ for arbitrary fixed
$\widehat{\boldsymbol{\beta}} \in \Re^{p}$, be a unique global minimum of the convex program
(\ref{equiv1}) is that, $\mathbf{x}^{*}(\widehat{\boldsymbol{\beta}})$ and the corresponding
Lagrange multiplier $\lambda^{*}(\widehat{\boldsymbol{\beta}}) \in \Re$, fulfill the
Kuhn-Tucker first order conditions
\begin{eqnarray}\label{kt1}
  \nabla_{\mathbf{x}}L(\mathbf{x}, \lambda; \widehat{\boldsymbol{\beta}}) =
  \left\{
    \begin{array}{l}
      \mathbf{M}(\mathbf{x}) \widehat{\boldsymbol{\beta}} + 2\lambda(\widehat{\boldsymbol{\beta}}) \mathbf{x}\\ \quad\mbox{or}\\
      \widehat{\boldsymbol{\beta}}_{1} + 2(\widehat{\mathbf{B}}+ \lambda(\widehat{\boldsymbol{\beta}}) \mathbf{I}_{n})\mathbf{x}
    \end{array}
  \right \}
   &=& \mathbf{0} \\
  \label{kt2}
  \nabla_{\lambda}L(\mathbf{x}, \lambda; \widehat{\boldsymbol{\beta}}) = ||\mathbf{x}||^{2} - c^{2} & \leq & 0 \\
  \label{kt3}
  \lambda(\widehat{\boldsymbol{\beta}})(||\mathbf{x}||^{2} - c^{2}) &=& 0\\
  \label{kt4}
  \lambda(\widehat{\boldsymbol{\beta}}) &\geq& 0
\end{eqnarray}
where
\begin{eqnarray*}
  \mathbf{M}(\mathbf{x}) &=& \nabla_{\mathbf{x}}\mathbf{z}'(\mathbf{x}) = \frac{\partial \mathbf{z}'(\mathbf{x})}{\partial \mathbf{x}} \\
     &=& (\mathbf{0}\vdots \mathbf{I}_{n} \vdots 2\diag(\mathbf{x}) \vdots \mathbf{C}_{1}\vdots\cdots
     \vdots\mathbf{C}_{n-1}) \in \Re^{n \times p},
\end{eqnarray*}
with
$$
  \mathbf{C}_{i} =
  \left(
     \begin{array}{c}
       \mathbf{0}'_{1} \\
       \vdots \\
       \mathbf{0}'_{i-1} \\
       \mathbf{x}'\mathbf{A}_{i} \\
       x_{i}\mathbf{I}_{n-i}
     \end{array}
  \right), i = 1, \dots, n-1, \quad \mathbf{0}_{j} \in \Re^{n-i}, j = 1, \dots, i-1,
$$
observing that when $i=1$ (i.e. j = 0) means that this row not appear in $\mathbf{C}_{1}$; and
$$
  \mathbf{A}_{i} =
  \left (
     \begin{array}{c}
       \mathbf{0}'_{1} \\
       \vdots \\
       \mathbf{0}'_{i} \\
       \mathbf{I}_{n-i}
     \end{array}
  \right), \quad \mathbf{0}'_{k} \in \Re^{n-i}, k = 1, \dots,i.
$$
\end{thm}
In addition, assume that strict complementarity slackness holds at
$\mathbf{x}^{*}(\boldsymbol{\beta})$ with respect to $\lambda^{*}(\boldsymbol{\beta})$, that is
\begin{equation}\label{a21}
    \lambda^{*}(\boldsymbol{\beta}) > 0 \Leftrightarrow ||\mathbf{x}||^{2} - c^{2} = 0.
\end{equation}
Analogously, the Khun-Tucker condition (\ref{kt1}) to (\ref{kt4}) for
$\widehat{\boldsymbol{\beta}} = \boldsymbol{\beta}$ are stated next.
\begin{cor}\label{cor1}
The necessary and sufficient conditions that a point $\mathbf{x}^{*}(\boldsymbol{\beta}) \in
\Re^{n}$ for arbitrary fixed $\boldsymbol{\beta} \in \Re^{p}$, be a unique global minimum of
the convex program (\ref{opp}) is that, $\mathbf{x}^{*}(\boldsymbol{\beta})$ and the
corresponding Lagrange multiplier $\lambda^{*}(\boldsymbol{\beta}) \in \Re$, fulfill the
Kuhn-Tucker first order conditions
\begin{eqnarray}\label{ktp1}
  \nabla_{\mathbf{x}}L(\mathbf{x}, \lambda; \widehat{\boldsymbol{\beta}}) =
  \left\{
    \begin{array}{l}
      \mathbf{M}(\mathbf{x}) \boldsymbol{\beta} + 2\lambda(\boldsymbol{\beta}) \mathbf{x}\\ \quad\mbox{or}\\
      \boldsymbol{\beta}_{1} + 2(\mathbf{B}\mathbf{x}+\lambda(\boldsymbol{\beta})\mathbf{I}_{n})\mathbf{x}
    \end{array}
  \right \}
   &=& \mathbf{0} \\
  \label{ktp2}
  \nabla_{\lambda}L(\mathbf{x}, \lambda; \boldsymbol{\beta}) = ||\mathbf{x}||^{2} - c^{2} & \leq & 0 \\
  \label{ktp3}
  \lambda(\boldsymbol{\beta})(||\mathbf{x}||^{2} - c^{2}) &=& 0\\
  \label{ktp4}
  \lambda(\boldsymbol{\beta}) &\geq& 0
\end{eqnarray}
and $\lambda(\boldsymbol{\beta}) = 0$ when $||\mathbf{x}||^{2} - c^{2} < 0$ at
$\left[\mathbf{x}^{*}(\boldsymbol{\beta}),\lambda^{*}(\boldsymbol{\beta})\right]'$.
\end{cor}
Observe that, due to the strict convexity of the constraint and objective function, the
second-order sufficient condition is evidently fulfilled for the convex program (\ref{equiv1}).

Next is established the existence of a once continuously differentiable solution to program
(\ref{equiv1}), see \citet{fg:82}.

\begin{thm}\label{teo2}
Assume that (\ref{a21}) hold true and the second-order sufficient condition is  for the convex
program (\ref{equiv1}). Then
\begin{enumerate}
  \item $\mathbf{x}^{*}(\boldsymbol{\beta})$ is a unique global minimum of program (\ref{opp})
  and $\lambda^{*}(\boldsymbol{\beta})$ is unique too.
  \item For $\widehat{\boldsymbol{\beta}} \in V_{\varepsilon}(\boldsymbol{\beta})$
  (is an $\varepsilon-$neighborhood or open ball), there exist a unique once continuously
  differentiable vector function
  $$
    \mathbf{r}^{*}(\widehat{\boldsymbol{\beta}}) =
    \left[
    \begin{array}{c}
      \mathbf{x}^{*}(\widehat{\boldsymbol{\beta}}) \\
      \lambda^{*}(\widehat{\boldsymbol{\beta}})
    \end{array}
    \right] \in \Re^{n+1}
  $$
  satisfying the second order sufficient conditions of problem (\ref{opp})such that $\mathbf{r}^{*}(\boldsymbol{\beta}) =
\left[\mathbf{x}^{*}(\boldsymbol{\beta}), \lambda^{*}(\boldsymbol{\beta})\right]'$ and hence,
$\mathbf{x}^{*}(\widehat{\boldsymbol{\beta}})$ is a unique global minimum of problem
(\ref{equiv1}) with associated unique Lagrange multiplier
$\lambda^{*}(\widehat{\boldsymbol{\beta}})$.
  \item For $\widehat{\boldsymbol{\beta}} \in V_{\varepsilon}(\boldsymbol{\beta})$, the status of the constraint is unchanged and
  $\lambda^{*}(\widehat{\boldsymbol{\beta}}) > 0 \Leftrightarrow ||\mathbf{x}||^{2} - c^{2} = 0$ is hold.
\end{enumerate}
\end{thm}

\section{Asymptotic normality of the critical point}\label{sec5}

This section considers  the statistical and mathematical programming aspects of the sensitivity
analysis of the optimum of a estimated response surface model.

\begin{thm} Assume:
\begin{enumerate}
   \item For any $\widehat{\boldsymbol{\beta}} \in V_{\varepsilon}(\boldsymbol{\beta})$,
   the second-order sufficient condition is fulfilled for the convex program (\ref{equiv1})
   such that the second order derivatives
   $$
     \frac{\partial^{2} L(\mathbf{x}, \lambda; \widehat{\boldsymbol{\beta}})}{\partial \mathbf{x}\partial
     \mathbf{x}'}, \  \frac{\partial^{2} L(\mathbf{x}, \lambda; \widehat{\boldsymbol{\beta}})}{\partial \mathbf{x}\partial
     \widehat{\boldsymbol{\beta}}'}, \  \frac{\partial^{2} L(\mathbf{x}, \lambda; \widehat{\boldsymbol{\beta}})}{\partial \mathbf{x}\partial
     \lambda}, \  \frac{\partial^{2} L(\mathbf{x}, \lambda; \widehat{\boldsymbol{\beta}})}{\partial \lambda \partial
     \mathbf{x}'}, \  \frac{\partial^{2} L(\mathbf{x}, \lambda; \widehat{\boldsymbol{\beta}})}{\partial \lambda\partial
     \widehat{\boldsymbol{\beta}}}
   $$
   exist and are continuous in $\left[\mathbf{x}^{*}(\widehat{\boldsymbol{\beta}}),
   \lambda^{*}(\widehat{\boldsymbol{\beta}})\right]' \in V_{\varepsilon}(\left[\mathbf{x}^{*}(\boldsymbol{\beta}),
   \lambda^{*}(\boldsymbol{\beta})\right]')$ and
   $$
     \frac{\partial^{2} L(\mathbf{x}, \lambda; \widehat{\boldsymbol{\beta}})}{\partial \mathbf{x}\partial
     \mathbf{x}'},
   $$
   is positive definite.
  \item $\widehat{\boldsymbol{\beta}}_{\nu}$ the estimator of the true parameter vector
  $\boldsymbol{\beta}_{\nu}$ that is based on a sample of size $N_{\nu}$ is such that
  $$
    \sqrt{N_{\nu}}(\widehat{\boldsymbol{\beta}}_{\nu} - \boldsymbol{\beta}_{\nu}) \sim
    \mathcal{N}_{p}(\mathbf{0}_{p}, \boldsymbol{\Sigma}), \quad
    \frac{1}{N_{\nu}}\boldsymbol{\Sigma} = \sigma^{2}(\mathbf{X}'\mathbf{X})^{-1}.
  $$
  \item (\ref{a21}) is fulfilled for $\widehat{\boldsymbol{\beta}} = \boldsymbol{\beta}$. Then
  asymptotically
  $$
    \sqrt{N_{\nu}}\left[\mathbf{x}^{*}(\widehat{\boldsymbol{\beta}}) -
     \mathbf{x}^{*}(\boldsymbol{\beta})\right] \build{\rightarrow}{d}{} \mathcal{N}_{n}(\mathbf{0}_{n}, \boldsymbol{\Xi})
  $$
  where the $n \times n$ variance-covariance matrix
  $$
    \boldsymbol{\Xi} = \left(\frac{\partial \mathbf{x}^{*}(\widehat{\boldsymbol{\beta}})}{\partial
    \widehat{\boldsymbol{\beta}}}\right)\widehat{\boldsymbol{\Sigma}} \left(\frac{\partial \mathbf{x}^{*}
    (\widehat{\boldsymbol{\beta}})}{\partial \widehat{\boldsymbol{\beta}}}\right)',
  $$
  such that all elements of $\left(\partial \mathbf{x}^{*}(\widehat{\boldsymbol{\beta}})/\partial
  \widehat{\boldsymbol{\beta}}\right)$ are continuous on any $\widehat{\boldsymbol{\beta}}
  \in V_{\varepsilon}(\boldsymbol{\beta})$; furthermore
  $$
    \left(\frac{\partial \mathbf{x}^{*}(\widehat{\boldsymbol{\beta}})}{\partial
    \widehat{\boldsymbol{\beta}}}\right) = \mathbf{G}^{-1}
    \left(\frac{\mathbf{x}^{*}(\widehat{\boldsymbol{\beta}})\mathbf{x}^{*}(\widehat{\boldsymbol{\beta}})'
    \mathbf{G}^{-1}}{\mathbf{x}^{*}(\widehat{\boldsymbol{\beta}})'
    \mathbf{G}^{-1}\mathbf{x}^{*}(\widehat{\boldsymbol{\beta}})}
    - \mathbf{I}_{n}\right)\mathbf{M}\left(\mathbf{x}^{*}(\widehat{\boldsymbol{\beta}})\right),
  $$
  where
  $$
    \mathbf{G} = \frac{\partial^{2} L(\mathbf{x}, \lambda; \widehat{\boldsymbol{\beta}})}{\partial \mathbf{x}\partial
     \mathbf{x}'}= 2\left(\widehat{\mathbf{B}}-\lambda^{*}
    (\widehat{\boldsymbol{\beta}})\mathbf{I}_{n}\right).
  $$
\end{enumerate}
\end{thm}
\begin{proof}
According to Theorem \ref{teo1} and Corollary \ref{cor1}, the
Kuhn-Tucker conditions (\ref{kt1})--(\ref{kt4}) at
$\left[\mathbf{x}^{*}(\widehat{\boldsymbol{\beta}}),\lambda^{*}(\widehat{\boldsymbol{\beta}})\right]'$
and to at (\ref{ktp1})--(\ref{ktp4})
$\left[\mathbf{x}^{*}(\boldsymbol{\beta}),\lambda^{*}(\boldsymbol{\beta})\right]'$
are fulfilled for mathematical programs (\ref{opp}) and
(\ref{equiv1}), respectively. From conditions
(\ref{ktp1})--(\ref{ktp4}) of Corollary \ref{cor1}, the following
system equation
\begin{eqnarray}\label{ktpp1}
  \nabla_{\mathbf{x}}L(\mathbf{x}, \lambda; \widehat{\boldsymbol{\beta}}) =
  \left\{
    \begin{array}{l}
      \mathbf{M}(\mathbf{x}) \boldsymbol{\beta} + 2\lambda(\boldsymbol{\beta}) \mathbf{x}\\ \quad\mbox{or}\\
      \boldsymbol{\beta}_{1} + 2(\mathbf{B}\mathbf{x}+\lambda(\boldsymbol{\beta})\mathbf{I}_{n})\mathbf{x}
    \end{array}
  \right \}
   &=& \mathbf{0}\\
  \label{ktpp2}
  \nabla_{\lambda}L(\mathbf{x}, \lambda; \boldsymbol{\beta}) = ||\mathbf{x}||^{2} - c^{2} & = & 0
\end{eqnarray}
has a solution $\mathbf{x}^{*}(\boldsymbol{\beta}), \lambda^{*}(\boldsymbol{\beta}) > 0,
\boldsymbol{\beta}$.

The nonsingular Jacobian matrix of the continuously differentiable functions (\ref{ktpp1}) and
(\ref{ktpp2}) with respect to $\mathbf{x}$ and $\lambda$ at
$\left[\mathbf{x}^{*}(\widehat{\boldsymbol{\beta}}),\lambda^{*}(\widehat{\boldsymbol{\beta}})\right]'$
is
$$
  \left(
     \begin{array}{cc}
       \displaystyle\frac{\partial^{2} L(\mathbf{x}, \lambda; \widehat{\boldsymbol{\beta}})}{\partial \mathbf{x}\partial
       \mathbf{x}'} & \displaystyle\frac{\partial^{2} L(\mathbf{x}, \lambda; \widehat{\boldsymbol{\beta}})}{\partial \lambda \partial
       \mathbf{x}} \\
       \displaystyle\frac{\partial^{2} L(\mathbf{x}, \lambda; \widehat{\boldsymbol{\beta}})}{\partial \mathbf{x}'\partial
       \lambda} & 0
     \end{array}
     \right )
   = \left(
     \begin{array}{cc}
       2(\widehat{\mathbf{B}}+\lambda\mathbf{x}) & 2\mathbf{x} \\
       2\mathbf{x}' & 0
     \end{array}
   \right).
$$
According to the implicit functions theorem, there is a neighborhood
$V_{\varepsilon}(\boldsymbol{\beta})$ such that for arbitrary $\widehat{\boldsymbol{\beta}} \in
V_{\varepsilon}(\boldsymbol{\beta})$, the system (\ref{ktpp1}) and (\ref{ktpp2}) has a unique
solution $\mathbf{x}^{*}(\widehat{\boldsymbol{\beta}})$,
$\lambda^{*}(\widehat{\boldsymbol{\beta}})$, $\widehat{\boldsymbol{\beta}}$ and by Theorem
\ref{teo2}, the components of $\mathbf{x}^{*}(\widehat{\boldsymbol{\beta}})$,
$\lambda^{*}(\widehat{\boldsymbol{\beta}})$ are continuously differentiable function of
$\widehat{\boldsymbol{\beta}}$, see \citet{bsh:74}. Their derivatives are given by
\begin{eqnarray}
  \left (
     \begin{array}{c}
       \displaystyle\frac{\partial \mathbf{x}^{*}(\widehat{\boldsymbol{\beta}})}{\partial
       \widehat{\boldsymbol{\beta}}} \\
       \displaystyle\frac{\partial \lambda^{*}(\widehat{\boldsymbol{\beta}})}{\partial
       \widehat{\boldsymbol{\beta}}}
     \end{array}
  \right )
  &=& - \left(
     \begin{array}{cc}
       \displaystyle\frac{\partial^{2} L(\mathbf{x}, \lambda; \widehat{\boldsymbol{\beta}})}{\partial \mathbf{x}\partial
       \mathbf{x}'} & \displaystyle\frac{\partial^{2} L(\mathbf{x}, \lambda; \widehat{\boldsymbol{\beta}})}{\partial \lambda \partial
       \mathbf{x}} \\
       \displaystyle\frac{\partial^{2} L(\mathbf{x}, \lambda; \widehat{\boldsymbol{\beta}})}{\partial \mathbf{x}'\partial
       \lambda} & 0
     \end{array}
     \right )^{-1}
     \left(
        \begin{array}{c}
          \displaystyle\frac{\partial^{2} L(\mathbf{x}, \lambda; \widehat{\boldsymbol{\beta}})}{\partial \mathbf{x}\partial
          \widehat{\boldsymbol{\beta}}'} \\
          0
        \end{array}
     \right)\nonumber\\
    \label{xx}
    &=& -
    \left(
     \begin{array}{cc}
       2(\widehat{\mathbf{B}}+\lambda\mathbf{x}) & 2\mathbf{x} \\
       2\mathbf{x}' & 0
     \end{array}
   \right)^{-1}
   \left (
      \begin{array}{c}
        \mathbf{M}(\mathbf{x}) \\
        0
      \end{array}
   \right ).
\end{eqnarray}
The explicit form of $(\partial \mathbf{x}^{*}(\widehat{\boldsymbol{\beta}})/ \partial
\widehat{\boldsymbol{\beta}})$ follow from (\ref{xx}) and by formula
$$
  \left(
    \begin{array}{cc}
      \mathbf{P} & \mathbf{Q} \\
      \mathbf{Q}' & \mathbf{0}
    \end{array}
  \right )^{-1}
  =
  \left(
    \begin{array}{cc}
      [\mathbf{I} - \mathbf{P}^{-1}\mathbf{Q}(\mathbf{Q}'\mathbf{P}^{-1}\mathbf{Q})^{-1}\mathbf{Q}']\mathbf{P}^{-1}
      & \mathbf{P}^{-1}\mathbf{Q}(\mathbf{Q}'\mathbf{P}^{-1}\mathbf{Q})^{-1} \\
      (\mathbf{Q}'\mathbf{P}^{-1}\mathbf{Q})^{-1}\mathbf{Q}'\mathbf{P}^{-1} & -(\mathbf{Q}'\mathbf{P}^{-1}\mathbf{Q})^{-1}
    \end{array}
  \right )
$$
where $\mathbf{P}$ and $\mathbf{S}$ are symmetric and $\mathbf{P}$ is nonsingular.

Then from assumption 2 and \citet[(iii), p. 388]{r:73} and \citet[Theorem 14.6-2, p.
493]{betal:91} (see also \citet[p. 353]{cr:46})
\begin{equation}\label{AN}
    \sqrt{N_{\nu}}\left[\mathbf{x}^{*}(\widehat{\boldsymbol{\beta}}) -
     \mathbf{x}^{*}(\boldsymbol{\beta})\right] \build{\rightarrow}{d}{} \mathcal{N}_{n}\left(\mathbf{0}_{n},
     \left(\frac{\partial \mathbf{x}^{*}(\boldsymbol{\beta})}{\partial
    \widehat{\boldsymbol{\beta}}}\right)\boldsymbol{\Sigma} \left(\frac{\partial \mathbf{x}^{*}
    (\boldsymbol{\beta})}{\partial \widehat{\boldsymbol{\beta}}}\right)'\right).
\end{equation}
Finally note that all elements of $(\partial \mathbf{x}^{*}/ \partial
\widehat{\boldsymbol{\beta}})$ are continuous on $V_{\varepsilon}(\boldsymbol{\beta})$, so that
the asymptotical distribution (\ref{AN}) can be substituted by
$$
  \sqrt{N_{\nu}}\left[\mathbf{x}^{*}(\widehat{\boldsymbol{\beta}}) -
     \mathbf{x}^{*}(\boldsymbol{\beta})\right] \build{\rightarrow}{d}{} \mathcal{N}_{n}\left(\mathbf{0}_{n},
     \left(\frac{\partial \mathbf{x}^{*}(\widehat{\boldsymbol{\beta}})}{\partial
    \widehat{\boldsymbol{\beta}}}\right)\widehat{\boldsymbol{\Sigma}} \left(\frac{\partial \mathbf{x}^{*}
    (\widehat{\boldsymbol{\beta}})}{\partial \widehat{\boldsymbol{\beta}}}\right)'\right),
$$
see \citet[(iv), pp.388--389]{r:73}. \qed
\end{proof}

Now, consider that $\lambda(\boldsymbol{\beta}) = 0$ that is, $||\mathbf{x}||^{2} - c^{2} < 0$
at $\mathbf{x}^{*}(\widehat{\boldsymbol{\beta}})$. Thus in this case we have an explicit
expression for $\mathbf{x}^{*}(\widehat{\boldsymbol{\beta}})$, furthermore
$$
  \mathbf{x}^{*}(\widehat{\boldsymbol{\beta}}) = -\frac{1}{2} \widehat{\mathbf{B}}^{-1}
  \widehat{\mathbf{b}}_{1}.
$$
Hence:
\begin{cor}\label{cor2}
Assume:
\begin{enumerate}
   \item For any $\widehat{\boldsymbol{\beta}} \in V_{\varepsilon}(\boldsymbol{\beta})$,
   the second-order sufficient condition is fulfilled for the convex program (\ref{equiv1})
   such that the second order derivatives
   $$
     \frac{\partial^{2} \widehat{y}(\mathbf{x})}{\partial \mathbf{x}\partial
     \mathbf{x}'}, \  \frac{\partial^{2} \widehat{y}(\mathbf{x})}{\partial \mathbf{x}\partial
     \widehat{\boldsymbol{\beta}}'},
   $$
   exist and are continuous in $\mathbf{x}^{*}(\widehat{\boldsymbol{\beta}}) \in V_{\varepsilon}(\mathbf{x}^{*}(\boldsymbol{\beta}))$ and
   $$
     \frac{\partial^{2} \widehat{y}(\mathbf{x})}{\partial \mathbf{x}\partial
     \mathbf{x}'},
   $$
   is positive definite.
  \item $\widehat{\boldsymbol{\beta}}_{\nu}$ the estimator of the true parameter vector
  $\boldsymbol{\beta}_{\nu}$ that is based on a sample of size $N_{\nu}$ is such that
  $$
    \sqrt{N_{\nu}}(\widehat{\boldsymbol{\beta}}_{\nu} - \boldsymbol{\beta}_{\nu}) \sim
    \mathcal{N}_{p}(\mathbf{0}_{p}, \boldsymbol{\Sigma}), \quad
    \frac{1}{N_{\nu}}\boldsymbol{\Sigma} = \sigma^{2}(\mathbf{X}'\mathbf{X})^{-1}.
  $$
  \item Then asymptotically
  $$
    \sqrt{N_{\nu}}\left[\mathbf{x}^{*}(\widehat{\boldsymbol{\beta}}) -
     \mathbf{x}^{*}(\boldsymbol{\beta})\right] \build{\rightarrow}{d}{} \mathcal{N}_{n}(\mathbf{0}_{n}, \boldsymbol{\Xi}_{1})
  $$
  where the $n \times n$ variance-covariance matrix
  $$
    \boldsymbol{\Xi}_{1} = \left(\frac{\partial \mathbf{x}^{*}(\widehat{\boldsymbol{\beta}})}{\partial
    \widehat{\boldsymbol{\beta}}}\right)\widehat{\boldsymbol{\Sigma}} \left(\frac{\partial \mathbf{x}^{*}
    (\widehat{\boldsymbol{\beta}})}{\partial \widehat{\boldsymbol{\beta}}}\right)',
  $$
  such that all elements of $\left(\partial \mathbf{x}^{*}(\widehat{\boldsymbol{\beta}})/\partial
  \widehat{\boldsymbol{\beta}}\right) \in \Re^{n \times p}$ are continuous on any $\widehat{\boldsymbol{\beta}}
  \in V_{\varepsilon}(\boldsymbol{\beta})$; furthermore
  \begin{eqnarray*}
    \displaystyle\left(\frac{\partial \mathbf{x}^{*}(\widehat{\boldsymbol{\beta}})}{\partial
    \widehat{\boldsymbol{\beta}}}\right) &=& \left(\frac{\partial^{2} \widehat{y}(\mathbf{x})}{\partial \mathbf{x}\partial
     \mathbf{x}'}\right)^{-1} \frac{\partial^{2} \widehat{y}(\mathbf{x})}{\partial \mathbf{x}\partial
     \widehat{\boldsymbol{\beta}}'}, \\
     &=& \frac{1}{2}\widehat{\mathbf{B}}^{-1}\mathbf{M}\left(\mathbf{x}^{*}(\widehat{\boldsymbol{\beta}})\right).
  \end{eqnarray*}
\end{enumerate}
\end{cor}
\begin{proof}
This is a verbatim copy of the proof of Theorem \ref{teo2}, noting that the conditions
(\ref{ktp1})--(\ref{ktp4}) are simply reduced to
$$
   \nabla_{\mathbf{x}} y(\mathbf{x}) =
  \left\{
    \begin{array}{l}
      \mathbf{M}(\mathbf{x}) \boldsymbol{\beta}\\ \quad\mbox{or}\\
      \boldsymbol{\beta}_{1} + 2\mathbf{B}\mathbf{x}
    \end{array}
  \right \} = \mathbf{0},
$$
which has the solution
$$\hspace{4cm}
  \mathbf{x}^{*}(\boldsymbol{\beta}) = -\frac{1}{2} \mathbf{B}^{-1}
  \mathbf{b}_{1}. \hspace{4cm}\mbox{\qed}
$$
\end{proof}




\section*{Conclusions}

As consequence of Theorem \ref{teo2} and Corollary \ref{cor2} now is feasible to establish
confidence intervals and hypothesis tests on the critical point, see \citet[Section 14.6.4, pp.
498--500]{betal:91}; then it is possible to establish operating conditions in intervals instead
of isolated points.

Unfortunately in many applications, especially in industry, the number of observations is
relatively small and perhaps the results obtained in this work should be applied with caution.
However, the results of this paper can be taken as a good first approximation to the exact
problem.

\section*{Acknowledgements}

The first author was partially supported IDI-Spain, Grants No. FQM2006-2271 and MTM2008-05785.
This paper was written during J. A. D\'{\i}az-Garc\'{\i}a's stay as a professor at the
Department of Mathematics of the University of Guanajuato.%


\begin{thebibliography}{99}

\bibitem[Aitchison and Silvey(1958)]{as:58}
    J. Aitchison and S.D. Silvey,
    Maximum likelihoocl estimation of parameters subject to restraints,
    Ann. Mathe. Statist., 29 (1958), 813--828.

\bibitem[Bigelow and Shapiro(1974)]{bsh:74}
    J. H. Gigelow, and N. Z. Shapiro,
    Implicit function theorem for mathematical progrtamming and for systems of iniqualities,
    Math. Programm. 6(2)(1974), 141--156.

\bibitem[Bishop \it{et al.}(1991)]{betal:91}
    Y. M. M. Bishop, S. E. Finberg, and P. W. Holland,
    Discrete Multivariate Analysis: Theory and Practice,
    The MIT press, Cambridge, 1991.

\bibitem[Chatterjee and Hadi(1988)]{chh:88}
    S. Chatterjee, and A. S.  Hadi,
    Sensitivity Analysis in Linear Regression,
    John Wiley: New York, 1988.

\bibitem[Cram\'er(1946)]{cr:46}
    H. Cram\'er,
    Mathematical Methods of Statistics,
    Princeton University Press, Princeton, 1946.

\bibitem[D\'{\i}az Garc\'{\i}a and Ramos-Quiroga(2001)]{dr:01}
    J. A. D\'{\i}az Garc\'{\i}a, and R. Ramos-Quiroga,
    An approach to  optimization in response surfaces,
    Comm. in Stat., Part A- Theory and Methods, 30(2001) 827--835.

\bibitem[D\'{\i}az Garc\'{\i}a and Ramos-Quiroga(2002)]{dre:02}
    J. A. D\'{\i}az Garc\'{\i}a, and R. Ramos-Quiroga,
    Erratum. An approach to  optimization in response surfaces,
    Comm. in Stat., Part A- Theory and Methods, 31(2002) 161.

\bibitem[Dupa\v{c}ov\'a(1984)]{d;84}
    J. Dupa\v{c}ov\'a,
    Stability in stochastic programming with recourse-estimated parameters,
    Math. Prog., 28(1984) 72--83.

\bibitem[Fiacco and Ghaemi(1982)]{fg:82}
    A. V. Fiacco and A. Ghaemi,
    Sensitivity analysis of a nonlinear structural design problem,
    Comp. and Ops. Res., 9(1)(1982) 29--55.

\bibitem[Jagannathan(1977)]{j:77}
    R. Jagannathan,
    Minimax procedure for a class of linear programs under uncertainty,
    Op. Res. 25(1977) 173--177.

\bibitem[Khuri and Cornell(1987)]{kc:87}
    A. I. Khuri, and J. A. Cornell,
    Response Surfaces: Designs and Analyses,
    Marcel Dekker, Inc., NewYork, 1987.

\bibitem[Myers \textit{et al.}(2009)]{mma:09}
    R. H. Myers, D. C. Montgomery, and C. M. Anderson-Cook,
    Response surface methodology: process and product optimization using designed
    experiments; third edition,
    Wiley, New York, 2009.

\bibitem[Rao(1973)]{r:73}
    C. R. Rao,
    Linear Statistical Inference and its Applications (2nd ed.),
    John  Wiley \& Sons, New York, 1973.

 \bibitem[Rao(1979)]{r:79}
    S. S. Rao,
    Optimization Theory and Applications,
    Wiley Eastern Limited,
    New Delhi, 1979.

\end{thebibliography}
\end{document}